\documentclass{article}

\usepackage{amsmath,amsthm,amssymb}
\usepackage{hyperref}
\usepackage{mathtools}
\usepackage{xcolor}

\allowdisplaybreaks

\definecolor{DarkPurple}{HTML}{381d2a}
\definecolor{QueenBlue}{HTML}{3e6990}
\definecolor{LaurelGreen}{HTML}{aabd8c}
\definecolor{DutchWhite}{HTML}{e9e3b4}
\definecolor{PinkOrange}{HTML}{f39b6d}

\hypersetup{
citebordercolor = LaurelGreen,
urlbordercolor = QueenBlue,
linkbordercolor = PinkOrange
}

\newtheorem{theorem}{Theorem}[section]
\newtheorem{lemma}[theorem]{Lemma}
\newtheorem{remark}[theorem]{Remark}

\newcommand{\calA}{\mathcal{A}}

\newcommand{\calC}{\mathcal{C}}

\newcommand{\calF}{\mathcal{F}}
\newcommand{\calG}{\mathcal{G}}

\newcommand{\calS}{\mathcal{S}}

\newcommand{\bbC}{\mathbb{C}}

\newcommand{\bbN}{\mathbb{N}}

\newcommand{\bbR}{\mathbb{R}}
\newcommand{\bbZ}{\mathbb{Z}}

\newcommand{\rme}{\mathrm{e}}
\newcommand{\rmi}{\mathrm{i}}

\newcommand{\set}[2]{\left\{ #1 \,\middle|\, #2 \right\}}
\newcommand{\abs}[1]{\left\lvert #1 \right\rvert}
\newcommand{\norm}[1]{\left\lVert #1 \right\rVert}

\DeclareMathOperator{\Exists}{\exists}

\renewcommand{\Re}{\operatorname{Re}}

\newcommand{\dd}{\,\mathrm{d}}

\title{Injectivity of sampled Gabor phase retrieval in spaces with general integrability conditions}
\author{Matthias Wellershoff\thanks{University of Maryland, Department of Mathematics, William E. Kirwan Hall, 4176 Campus Drive, College Park, MD 20742, \href{mailto:wellersm@umd.edu}{\texttt{wellersm@umd.edu}}}}
\date{\today}

\begin{document}
\maketitle

\begin{abstract}
    It was recently shown that functions in $L^4([-B,B])$ can be uniquely recovered up to a global phase factor from the absolute values of their Gabor transforms sampled on a rectangular lattice. We prove that this remains true if one replaces $L^4([-B,B])$ by $L^p([-B,B])$ with $p \in [1,\infty]$. To do so, we adapt the original proof by Grohs and Liehr and use a classical sampling result due to Beurling. Furthermore, we present a minor modification of a result of M{\"u}ntz--Sz{\'a}sz type by Zalik. Finally, we consider the implications of our results for more general function spaces obtained by applying the fractional Fourier transform to $L^p([-B,B])$ and for more general nonuniform sampling sets.

    \vspace{5pt}

    \vspace{5pt}
\end{abstract}

\section{Introduction}

In this paper, we consider the \emph{Gabor transform} of functions $f \in L^p(\bbR)$, $p \in [1,\infty]$, given by
\begin{equation*}
    \calG f (x,\omega) := 2^{1/4} \int_\bbR f(t) \rme^{-\pi(t-x)^2} \rme^{-2\pi\rmi t \omega} \dd t, \qquad (x,\omega) \in \bbR^2,
\end{equation*}
and try to understand if one can recover $f$ from measurements of the absolute value $\abs{\calG f}$ on discrete sets $S \subset \bbR^2$. This so-called sampled Gabor phase retrieval problem has recently been studied extensively \cite{alaifari2020phase,alaifari2021uniqueness,grohs2021foundational,grohs2023injectivity}. It is an elegant mathematical problem in the sense that it is rather easy to state while, at the same time, being less easy to solve. Moreover, it is connected to certain audio processing applications such as the phase vocoder \cite{flanagan1966phase,pruuvsa2017phase}.

A hallmark of all phase retrieval problems is that signals cannot be fully recovered from phaseless measurements. For the Gabor phase retrieval problem, we can see that the functions $f$ and $\rme^{\rmi \alpha} f$, where $\alpha \in \bbR$, generate the same measurements 
\begin{equation*}
    \abs{\calG (\rme^{\rmi \alpha} f)} = \abs{\rme^{\rmi \alpha} \calG f} = \abs{\calG f}.
\end{equation*}
Hence, we are not able to distinguish between $f$ and $\rme^{\rmi \alpha} f$ on the basis of their Gabor transform magnitude samples. We will therefore consider the equivalence relation $\sim$ on $L^p(\bbR)$ defined by 
\begin{equation}
    \label{eq:equivalence}
    f \sim g :\iff \Exists \alpha \in \bbR : f = \rme^{\rmi \alpha} g.
\end{equation}
With the help of this relation, we can introduce the phase retrieval operator $\calA : V/{\sim} \to [0,\infty)^{S}$, where $V$ is a subspace of $L^p(\bbR)$, by 
\begin{equation*}
    \calA(f)(x,\omega) := \abs{\calG f(x,\omega)}, \qquad (x,\omega) \in S,
\end{equation*}
for $f \in V/{\sim}$. The \emph{sampled Gabor phase retrieval problem} is the problem of inverting $\calA$ when $S \subset \bbR^2$ is discrete.

We note that it has long been known that one can invert $\calA$ for $V = L^2(\bbR)$ and $S = \bbR^2$. In applications, one does typically not have access to measurements of the Gabor transform magnitude on the entire time-frequency plane, however, and we thus believe that the sampled Gabor phase retrieval problem is a natural first step towards a better understanding of settings encountered in practice.

Relatively little was known about the inversion of $\calA$ for discrete sets $S$ until recently when a series of breakthroughs was presented in the papers \cite{alaifari2020phase,alaifari2021uniqueness,grohs2021foundational,grohs2023injectivity}. For the genesis of this paper, the work in \cite{grohs2023injectivity} was most important. The authors of that paper show that sampled Gabor phase retrieval is unique when $V = L^4([-B,B])$ and $S = \bbZ \times (4B)^{-1} \bbZ$. The assumption that signals lie in $L^4([-B,B])$ does not seem very natural, however, such that we are interested in extending the result to spaces with more general integrability conditions, and notably to $L^2([-B,B])$ as well as $L^1([-B,B])$. We will do so here and prove the following main result.

\begin{theorem}
    \label{thm:simplemaintheorem}
    Let $p \in [1,\infty]$, $B > 0$ and $b \in (0,\tfrac{1}{4B})$. Then, the following are equivalent for $f,g \in L^p([-B,B])$:
    \begin{enumerate}
        \item $f = \rme^{\rmi \alpha} g$ for some $\alpha \in \bbR$, 
        \item $\abs{\calG f} = \abs{\calG g}$ on $\bbN \times b \bbZ$.
    \end{enumerate}
\end{theorem}

We observe that this theorem is almost optimal in view of the results presented in \cite{alaifari2020phase}: there, for any lattice $S \subset \bbR^2$ in the time-frequency plane, explicit examples $f,g \in L^2(\bbR)$ were constructed which do not agree up to global phase but which satisfy that
\begin{equation*}
    \abs{\calG f (x,\omega)} = \abs{\calG g (x,\omega)}, \qquad (x,\omega) \in S.
\end{equation*} 
In particular, it is necessary to restrict the Gabor phase retrieval problem to a proper subspace $V$ of $L^2(\bbR)$ in order to obtain a uniqueness result from samples. It may not surprise the reader that one may further generalise Theorem \ref{thm:simplemaintheorem} in multiple ways to include more general function spaces obtained by taking fractional Fourier transforms of elements in $L^p([-B,B])$ or more general nonuniform sampling sets. Both of these generalisation have already been suggested in \cite{grohs2023injectivity} and we adapt them here. We remark that during the review process for this paper, other manuscripts addressing aspects of sampled short-time Fourier transform phase retrieval have emerged, such as \cite{alaifari2022connection, grohs2022completeness, iwen2023phase, wellershoff2023sampling}.

\paragraph{Outline} In Section~\ref{sec:basicnotions}, we introduce some basic concepts needed for the further understanding of this paper. Specifically, we introduce the fractional Fourier transform and the Paley--Wiener spaces along with some of their most relevant properties. Additionally, we provide two core insights: the first of those being that the Gabor transform of a bandlimited function is bandlimited in the first argument (cf.~Lemma~\ref{lem:acha_masterlockI}) and the second of those being a result of M{\"u}ntz--Sz{\'a}sz type by Zalik (cf.~Theorem~\ref{thm:zalik_full}).

In Section~\ref{sec:mainresults}, we prove our main result, which states that for sufficiently dense uniformly discrete sets $X$ and certain countable sets $\Omega$, functions in $\mathrm{PW}_B^p$, $p \in [1,\infty]$, are uniquely determined (up to global phase) from their Gabor magnitudes on $X \times \Omega$. We finally extend this result by applying the fractional Fourier transform to rotate the time-frequency plane.

\paragraph{Notation} We use the convention 
\begin{equation*}
    \calF f (\xi) = \int_{\bbR^d} f(t) \rme^{-2\pi\rmi (t, \xi)} \dd t, \qquad \xi \in \bbR^d,
\end{equation*}
for the Fourier transform on the Schwartz space $\calS(\bbR^d)$. It is well-known that the Fourier transform can be extended to the space of tempered distributions $\calS'(\bbR^d)$ and thereby to all $L^p(\bbR^d)$, $p \in [1,\infty]$. The inverse of the Fourier transform is given by 
\begin{equation*}
    \calF^{-1} g (t) = \int_{\bbR^d} g(\xi) \rme^{2\pi\rmi (\xi, t)} \dd \xi, \qquad t \in \bbR^d,
\end{equation*}
for $g \in \calS(\bbR^d)$ and can likewise be extended to $\calS'(\bbR^d)$. We use the translation operators $\{\operatorname{T}_x\}_{x \in \bbR}$ given by 
\begin{equation*}
    \operatorname{T}_x f (t) := f(t-x), \qquad t \in \bbR,
\end{equation*} 
for $x \in \bbR$, as well as the modulation operators $\{\operatorname{M}_\omega\}_{\omega \in \bbR}$ given by 
\begin{equation*}
    \operatorname{M}_\omega f (t) := \rme^{2\pi\rmi t \omega} f(t), \qquad t \in \bbR,
\end{equation*} 
for $\omega \in \bbR$. We denote the normalised Gaussian by $\phi(t) = 2^{1/4} \rme^{-\pi t^2}$ where $t \in \bbR$. The cardinality of a finite set $X \subset \bbR$ is denoted by $\abs{X}$. Finally, we denote the reflection and complex conjugation of functions by $f^\# (t) := \overline{f(-t)}$ for $t \in \bbR$, and the duality pairing between the space of tempered distributions $\calS'(\bbR^d)$ and the Schwartz space $\calS(\bbR^d)$ by $\langle f,\psi \rangle$, where $f \in \calS'(\bbR^d)$ and $\psi \in \calS(\bbR^d)$. With this, we can define the Gabor transform of a tempered distribution $f \in \calS'(\bbR)$ via $\calG f (x,\omega) := \langle f, \operatorname{M}_\omega \operatorname{T}_x \phi\rangle$.

\section{Preliminaries}
\label{sec:basicnotions}

\subsection{The fractional Fourier transform}

The \emph{fractional Fourier transform} of a function $f \in \calS(\bbR)$ is defined by
\begin{equation*}
    \calF_\theta f(\xi) := c_\theta \rme^{\pi \rmi \xi^2 \cot \theta} \int_\bbR f(t) \rme^{\pi\rmi t^2 \cot \theta} \rme^{-2 \pi \rmi \frac{t \xi}{\sin \theta}} \dd t, \qquad \xi \in \bbR,
\end{equation*}
for $\theta \in \bbR \setminus \pi \bbZ$, where $c_\theta \in \bbC$ is the square root of $1 - \rmi \cot \theta$ with positive real part, and by $\calF_{2k \pi} f := f$ as well as $\calF_{(2k+1)\pi} f(\xi) := f(-\xi)$, for $\xi \in \bbR$, where $k \in \bbZ$. One can show that the fractional Fourier transform satisfies \cite[Theorem~2.1 on p.~406]{kerr1988namias}
\begin{equation*}
    \int_{\bbR^d} \calF_\theta f (\xi) g(\xi) \dd \xi = \int_{\bbR^d} f (t) \calF_\theta g(t) \dd t, \qquad f,g \in \calS(\bbR)
\end{equation*}
and that it maps a Schwartz function to a Schwartz function \cite[Theorem~5.3 on p.~170]{mcbride1987namias}. Therefore, we can extend the fractional Fourier transform to the tempered distributions via 
\begin{equation*}
    \langle \calF_\theta f, \psi \rangle := \langle f, \calF_\theta \psi \rangle, \qquad \psi \in \calS(\bbR),
\end{equation*}
for $f \in \calS'(\bbR)$.

The fractional Fourier transform is a powerful tool in time-frequency analysis. One of its most crucial properties is that it corresponds to a rotation of the time-frequency plane \cite{almeida1994fractional}. To describe this property, we introduce the rotation operator $\operatorname{R}_\theta : \bbR^2 \to \bbR^2$, defined by $\operatorname{R}_\theta(x,\omega) := (x \cos \theta - \omega \sin \theta, x \sin \theta + \omega \cos \theta)$, where $\theta \in \bbR$ and $x,\omega \in \bbR$.

\begin{lemma}
    \label{lem:frftandgabor}
    Let $\theta \in \bbR$ and $f \in \calS'(\bbR)$. It holds that 
    \begin{equation*}
        \calG \calF_\theta f (x,\omega) = \rme^{\pi \rmi \left(x^2 - \omega^2\right) \sin \theta \cos \theta + 2\pi\rmi x \omega \sin^2 \theta} \cdot \calG f (\operatorname{R}_{-\theta}(x,\omega)), \qquad (x,\omega) \in \bbR^2.
    \end{equation*}
\end{lemma}

\begin{proof}
    Let $(x,\omega) \in \bbR^2$ be arbitrary but fixed. We consider 
    \begin{equation*}
        \calG \calF_\theta f (x,\omega) = \langle \calF_\theta f , \operatorname{M}_\omega \operatorname{T}_x \phi \rangle
        = \langle f , \calF_\theta \operatorname{M}_\omega \operatorname{T}_x \phi \rangle
    \end{equation*}
    such that the lemma boils down to understanding the commutative properties of the fractional Fourier transform and the modulation and translation operators. According to \cite[Table~I on p.~3086]{almeida1994fractional}, it holds that
    \begin{gather*}
        \calF_\theta \operatorname{T}_x = \rme^{\pi \rmi x^2 \sin \theta \cos \theta} \cdot \operatorname{M}_{-x \sin \theta} \operatorname{T}_{x \cos \theta} \calF_\theta, \\
        \calF_\theta \operatorname{M}_\omega = \rme^{-\pi \rmi \omega^2 \sin \theta \cos \theta} \cdot \operatorname{M}_{\omega \cos \theta} \operatorname{T}_{\omega \sin \theta} \calF_\theta.
    \end{gather*}
    Using that the fractional Fourier transform of the Gaussian is the Gaussian, we therefore find 
    \begin{align*}
        \calF_\theta \operatorname{M}_\omega \operatorname{T}_x \phi  &= \rme^{\pi \rmi \left(x^2 - \omega^2\right) \sin \theta \cos \theta} \cdot \operatorname{M}_{\omega \cos \theta} \operatorname{T}_{\omega \sin \theta} \operatorname{M}_{-x \sin \theta} \operatorname{T}_{x \cos \theta} \phi \\
        &= \rme^{\pi \rmi \left(x^2 - \omega^2\right) \sin \theta \cos \theta + 2\pi\rmi x \omega \sin^2 \theta} \cdot \operatorname{M}_{\omega \cos \theta - x \sin \theta} \operatorname{T}_{x \cos \theta + \omega \sin \theta} \phi
    \end{align*}
    which proves the lemma.
\end{proof}

\subsection{The Paley--Wiener spaces}
\label{ssec:paley--wiener}

In the following, we will mostly work with bandlimited functions. Precisely, we consider the \emph{Paley--Wiener spaces} defined by
\[
    \mathrm{PW}^p_B := \set{ f \in \calS'(\bbR) }{ f = \calF F \mbox{ for some } F \in L^p([-B,B]) },
\]
for $B>0$ and $p \in [1,\infty]$. One may see that the Paley--Wiener spaces are nested --- $\mathrm{PW}_B^q \subset \mathrm{PW}_B^p$ for $1 \leq p \leq q \leq \infty$ --- which is due to the nestedness of $L^p$-spaces over closed intervals --- $L^q([-B,B]) \subset L^p([-B,B])$ for $1 \leq p \leq q \leq \infty$. It therefore follows that $\mathrm{PW}_B^p \subset \mathrm{PW}_B^2 \subset L^2(\bbR)$ for $p \in [2,\infty]$ since the Fourier transform is unitary on $L^2(\bbR)$. Additionally, $\mathrm{PW}_B^p \subset L^{q}(\bbR)$ for $p \in [1,2]$, where $q \in [2,\infty]$ is the H{\"o}lder conjugate of $p$, by the Hausdorff--Young inequality. Hence, the elements of Paley--Wiener spaces are $L^p$-functions. In fact, one can show that the elements of Paley--Wiener spaces extend to entire functions of exponential-type. This is one direction of the well-known Paley--Wiener the-
orem.

An important property of bandlimited functions is that one may recover them from samples. Let us call a subset $X \subset \bbR$ \emph{uniformly discrete} if there exists an $\epsilon > 0$ such that, for all $x,y \in X$ with $x \neq y$, it holds that $\abs{x-y} > \epsilon$. We say that $X$ is a \emph{set of uniqueness} for $\mathrm{PW}_B^p$ if 
\begin{equation*}
    f(x) = 0 \mbox{ for all } x \in X \implies f = 0,
\end{equation*}
for $f \in \mathrm{PW}_B^p$. Similarly, we say that $X$ is a \emph{set of sampling} for $\mathrm{PW}_B^1$ if there exists a constant $K > 0$ such that
\begin{equation*}
    \sup_{t \in \bbR} \abs{f(t)} \leq K \sup_{x \in X} \abs{f(x)}, \qquad f \in \mathrm{PW}_B^1.
\end{equation*}
Clearly, every set of sampling for $\mathrm{PW}_B^1$ is a set of uniqueness for $\mathrm{PW}_B^p$ by the nestedness of the Paley--Wiener spaces. Hence, $f \in \mathrm{PW}_B^p$ is uniquely determined by $(f(x))_{x \in X}$ if $X$ is a set of sampling for $\mathrm{PW}_B^1$. Sets of sampling for $\mathrm{PW}_B^1 \subset C_0(\bbR)$ were characterised by Beurling in a series of seminar lectures given at Princeton \cite{beurling1989works}. Specifically, Beurling introduces
\begin{equation*}
    \underline n (r) := \inf_{t \in \bbR} \abs{ X \cap [t,t+r] }, \qquad r > 0,
\end{equation*} 
along with the \emph{lower uniform density}
\begin{equation*}
    \mathrm{l.u.d.}(X) := \lim_{r \to \infty} \frac{\underline n (r)}{r}
\end{equation*}
and proves the following result.

\begin{theorem}[{\cite[Theorem~5 on p.~346]{beurling1989works}}]\label{thm:beurling_sampling}
    A uniformly discrete set $X \subset \bbR$ is a set of sampling for $\mathrm{PW}_B^1$ if and only if 
    \begin{equation*}
        \mathrm{l.u.d.}(X) > 2B.
    \end{equation*}
\end{theorem}

\begin{remark}
    Notably, there are uniformly discrete sets $X$ that form a set of uniqueness for $\mathrm{PW}_B^2$ with $\mathrm{l.u.d.}(X) = 0$ \cite{koosis1960totalite}. This illustrates that $\mathrm{l.u.d.}(X) > 2B$ is sufficient but not necessary for $X$ being a set of uniqueness for $\mathrm{PW}_B^p$ in general.
\end{remark}

Finally, we note that the Gabor transform of a bandlimited function is bandlimited itself in the time variable (after modulation) and that therefore the square of the Gabor transform magnitudes is bandlimited.

\begin{lemma}\label{lem:acha_masterlockI}
    Let $p \in [1,\infty]$, $B > 0$, $\omega \in \bbR$ and $f \in \mathrm{PW}_B^p$. Then, $x \mapsto \rme^{2\pi\rmi x \omega} \calG f(x,\omega) \in \mathrm{PW}_B^p$ and therefore $x \mapsto \lvert \calG f (x,\omega) \rvert^2 \in \mathrm{PW}_{2B}^q$, where $q \in [1,\infty]$ is such that 
    \begin{equation*}
        1 + \frac{1}{q} = \frac{2}{p}.
    \end{equation*}
\end{lemma}

\begin{proof}
    First, we write $\calG f(x,\omega)$ in terms of the inverse Fourier transform of the signal $\calF^{-1} f$ and the Gaussian $\phi$ (cf.~\cite[Lemma~3.1.1 on p.~39]{grochenig2001foundations}):
    \begin{equation*}
        \rme^{2 \pi \rmi x \omega} \calG f(x,\omega) = \calF\left( \calF f \cdot \operatorname{T}_\omega \phi \right)(-x) = \calF\left( \calF^{-1} f \cdot \operatorname{T}_{-\omega} \phi \right)(x).
    \end{equation*}
    Next, we use the facts that $\calF^{-1} f \in L^p([-B,B])$ and $\phi \in L^\infty(\bbR)$ to see that $\calF^{-1} f \cdot \operatorname{T}_{-\omega} \phi$ is also in $L^p([-B,B])$. This implies that $x \mapsto \rme^{2 \pi \rmi x \omega} \calG f(x,\omega)$ is in the Paley-Wiener space $\mathrm{PW}_B^p$.

    We then compute the squared modulus of $\calG f(x,\omega)$ as a convolution of two functions in Fourier space: specifically, let $F_\omega := \calF^{-1} f \cdot \operatorname{T}_{-\omega} \phi$. Then, 
    \begin{equation*}
        \lvert \calG f(x,\omega) \rvert^2 = \calF F_\omega (x) \cdot \overline {\calF F_\omega (x)} = \calF F_\omega (x) \cdot \calF F_\omega^\#(x) = \calF \left(F_\omega \ast F_\omega^\# \right)(x).
    \end{equation*}
    Using Young's convolution inequality, we can show that $F \ast F_\omega^\# \in L^q(\bbR)$, where $q \in [1,\infty]$ satisfies
    \begin{equation*}
        1 + \frac{1}{q} = \frac{2}{p}.
    \end{equation*}
    Finally, we can bound the support of $F \ast F_\omega^\#$ by $[-2B,2B]$. Putting everything together, we conclude that $x \mapsto \lvert \calG f(x,\omega) \rvert^2$ is in the Paley-Wiener space $\mathrm{PW}_{2B}^q$.
\end{proof}

We have shown that $x \mapsto \abs{\calG f (x,\omega)}^2$ belongs to the Paley--Wiener space $\mathrm{PW}_{2B}^q \subset \mathrm{PW}_{2B}^1$. Therefore, we can apply Theorem~\ref{thm:beurling_sampling} to conclude that $(\abs{\calG f (x,\omega)})_{x \in X}$ completely determines $x \mapsto \abs{\calG f (x,\omega)}$, for $\omega \in \bbR$ fixed, provided that $X \subset \bbR$ is a uniformly discrete set such that $\mathrm{l.u.d.}(\Lambda) > 4B$.

\begin{remark}
    A similar result can be proven for the short-time Fourier transform with window $\psi \in \calF L^q(\bbR)$, where $q \in [1,\infty]$, by following the same proof strategy. In this way, one may arrive at a partial sampling result for short-time Fourier transform phase retrieval in which one assumes knowledge of measurements on $X \times \bbR$, where $X$ is a set of uniqueness for the appropriate Paley--Wiener space.
\end{remark}

\subsection{Zalik's theorem}
\label{sec:zalik}

For the proof of our main result, we need the following result of M{\"u}ntz--Sz{\'a}sz type due to Zalik. It asserts that certain translates of Gaussians are complete in the spaces $L^p([a,b])$ and $C([a,b])$. 

\begin{theorem}[{Zalik's theorem; \cite[Theorem~4 on p.~302]{zalik1978approximation}}]
    \label{thm:zalik_full}
    Let $p \in [1,\infty)$, $-\infty < a < b < \infty$, $c > 0$, and let $\Omega \subset \bbR$ be a countable set. Then, 
    \begin{equation*}
        \set{t \mapsto \rme^{-c (t-\omega)^2}}{ \omega \in \Omega }
    \end{equation*}
    is complete in $L^p([a,b])$ and $C([a,b])$ if and only if 
    \begin{equation}\label{eq:zaliks_condition}
        \sum_{\omega \in \Omega \setminus \{0\}} \abs{\omega}^{-1}
    \end{equation}
    diverges.
\end{theorem}

\begin{proof}
    The result for $L^p([a,b])$ follows from a small modification of the original proof in \cite{zalik1978approximation}. For $C([a,b])$ consider the following argument which is also inspired by the original proof: suppose that the sum in equation~\eqref{eq:zaliks_condition} diverges, and let $f \in \calC([a,b])$ as well as $\epsilon > 0$ be arbitrary but fixed. Define
    \begin{equation*}
        g(x) := f\left( \frac{\log x}{2 c} \right) \rme^{\frac{\log^2 x}{4 c}}, \qquad x \in \left[\rme^{2ca},\rme^{2cb}\right],
    \end{equation*}
    and note that $g \in C([\rme^{2ca},\rme^{2cb}])$. According to \cite[Theorem~6.1~on p.~30]{luxemburg1971entire}, $\{ x \mapsto x^{\omega} \,|\, \omega \in \Omega \}$ is complete in $C([\rme^{2ca},\rme^{2cb}])$. Therefore, there exist $\Omega_0 \subset \Omega$ finite and $(\lambda_\omega)_{\omega \in \Omega_0} \in \bbC$ such that 
    \begin{equation*}
        \sup_{x \in [\rme^{2ca},\rme^{2cb}]} \abs{ \sum_{\omega \in \Omega_0} \lambda_\omega x^\omega - g(x) } < \epsilon. 
    \end{equation*}
    By the change of variable $x = \rme^{2 c t}$, it follows that 
    \begin{equation*}
        \sup_{t \in [a,b]} \abs{ \sum_{\omega \in \Omega_0} \lambda_\omega \rme^{2 c \omega t} - f(t) \rme^{c t^2} } < \epsilon
    \end{equation*}
    and therefore 
    \begin{equation*}
        \sup_{t \in [a,b]} \abs{ \sum_{\omega \in \Omega_0} \lambda_\omega \rme^{c\omega^2} \cdot \rme^{- c(t - \omega)^2} - f(t) } \leq \sup_{t \in [a,b]} \abs{ \sum_{\omega \in \Omega_0} \lambda_\omega \rme^{2 c \omega t} - f(t) \rme^{c t^2} } < \epsilon
    \end{equation*}
    showing that $\{t \mapsto \rme^{-c (t-\omega)^2} \,|\, \omega \in \Omega \}$ is complete in $C([a,b])$.
    
    Now, suppose that $\{t \mapsto \rme^{-c (t-\omega)^2} \,|\, \omega \in \Omega\}$ is complete in $C([a,b])$. Then, a modification of the argument above shows that $\{ x \mapsto x^{\omega} \,|\, \omega \in \Omega \}$ is complete in $C([\rme^{2ca},\rme^{2cb}])$. Therefore, \cite[Theorem~6.1~on p.~30]{luxemburg1971entire} implies that the sum in equation~\eqref{eq:zaliks_condition} diverges.
\end{proof}

\begin{remark}
    By the same proof, we can extend Zalik's theorem to a countable subset $\Omega$ of $\bbC$ under the condition that there exists a $\delta > 0$ and a finite subset $\Omega_0 \subset \Omega$ such that
    \begin{equation*}
        \abs{\Re\left(\omega - \frac12\right)} \geq \delta \abs{ \omega - \frac12 }, \qquad \omega \in \Omega \setminus \Omega_0.
    \end{equation*}
\end{remark}

\section{Main results}
\label{sec:mainresults}

After laying out the groundwork, we present our main theorem, which generalises a result previously established by Grohs and Liehr for $\mathrm{PW}_B^4$ \cite{grohs2023injectivity}. Specifically, we extend their result to hold for the entire family of Paley--Wiener spaces $\mathrm{PW}_B^p$ for $p \in [1, \infty]$. This encompasses the case $p=2$ which is commonly studied in signal processing, as well as the challenging case $p=1$. Our result shows that for sufficiently dense uniformly discrete sets $X$ and certain countable sets $\Omega$, functions in $\mathrm{PW}_B^p$ are uniquely determined (up to global phase) from their Gabor magnitudes on $X \times \Omega$.

\begin{theorem}[Main result]\label{thm:mainthm}
    Let $p \in [1,\infty]$ and $B > 0$. Let $X \subset \bbR$ be a uniformly discrete set with $l.u.d.(X) > 4B$ and let $\Omega \subset \bbR$ be a countable set such that $\sum_{\omega \in \Omega \setminus \{0\}} \abs{\omega}^{-1}$ diverges. Then, the following are equivalent for $f,g \in \mathrm{PW}_B^p$:
    \begin{enumerate}
        \item $f = \rme^{\rmi \alpha} g$ for some $\alpha \in \bbR$,
        \item $\abs{\calG f} = \abs{\calG g}$ on $X \times \Omega$.
    \end{enumerate}
\end{theorem}

\begin{proof}[Proof of Theorem~\ref{thm:mainthm}]
    The following arguments are inspired by \cite{grohs2023injectivity}. It is obvious that item~1 implies item~2. Therefore, we assume that $f,g \in \mathrm{PW}_B^p$ are such that $\abs{\calG f} = \abs{\calG g}$ on $X \times \Omega$. Let us now fix an arbitrary $\omega \in \Omega$ and note that Lemma~\ref{lem:acha_masterlockI} implies that $x \mapsto \lvert \calG f(x,\omega) \rvert^2, x \mapsto \lvert \calG g(x,\omega) \rvert^2 \in \mathrm{PW}_{2B}^1$. It follows by Theorem~\ref{thm:beurling_sampling} that
    \begin{equation*}
        \abs{\calG f(x,\omega)} = \abs{\calG g(x,\omega)}, \qquad x \in \bbR.
    \end{equation*}

    Following the argument in the proof of Lemma~\ref{lem:acha_masterlockI}, it is readily shown that 
    \begin{equation*}
        \left( \calF^{-1} f \cdot \operatorname{T}_{-\omega} \phi \right) \ast \left(\calF^{-1} f \cdot \operatorname{T}_{-\omega} \phi \right)^\# = \left( \calF^{-1} g \cdot \operatorname{T}_{-\omega} \phi \right) \ast \left(\calF^{-1} g \cdot \operatorname{T}_{-\omega} \phi \right)^\#
    \end{equation*}
    in $L^1([-2B,2B])$. We reformulate the above to 
    \begin{multline*}
        \int_{-B}^B \calF^{-1} f (\eta) \overline{ \calF^{-1} f (\eta - \xi) } \phi(\eta + \omega) \phi(\eta- \xi + \omega) \dd \eta \\= \int_{-B}^B \calF^{-1} g (\eta) \overline{ \calF^{-1} g (\eta - \xi) } \phi(\eta + \omega) \phi(\eta- \xi + \omega) \dd \eta
    \end{multline*}
    for almost all $\xi \in \bbR$. By completing the square, we see that 
    \begin{equation*}
        \phi(\eta + \omega) \phi(\eta- \xi + \omega) = \sqrt{2} \cdot \rme^{-2\pi(\eta + \omega - \xi/2)^2} \cdot \rme^{-\frac{\pi \xi^2}{2}}
    \end{equation*}
    and therefore 
    \begin{equation}\label{eq:hebel_zalik}
        \int_{-B}^B \left( \calF^{-1} f (\eta) \overline{ \calF^{-1} f (\eta - \xi) } - \calF^{-1} g (\eta) \overline{ \calF^{-1} g (\eta - \xi) } \right) \rme^{-2\pi(\eta + \omega - \xi/2)^2} \dd \eta = 0,
    \end{equation}
    for almost all $\xi \in \bbR$. According to Fubini's theorem,
    \begin{equation*}
        \eta \mapsto H_\xi(\eta) := \calF^{-1} f (\eta) \overline{ \calF^{-1} f (\eta - \xi) } - \calF^{-1} g (\eta) \overline{ \calF^{-1} g (\eta - \xi)} \in L^1([-B,B])
    \end{equation*} 
    for almost all $\xi \in \bbR$. Let us now fix an arbitrary $\xi \in \bbR$ such that the above two assertions hold and consider the set $\Omega^\xi := \tfrac{\xi}{2} - \Omega$. Then, $\Omega^\xi \subset \bbR$ is countable and 
    \begin{equation*}
        \sum_{\omega \in \Omega^\xi \setminus \{0\}} \abs{\omega}^{-1}
    \end{equation*}
    diverges. Therefore, Theorem~\ref{thm:zalik_full} implies that the set
    \begin{equation*}
        \set{\eta \mapsto \rme^{-2\pi (\eta-\omega)^2}}{ \omega \in \Omega^\xi }
    \end{equation*}
    is complete in $C([-B,B])$. It follows that, for all $t \in \bbR$ and all $\epsilon > 0$, there exist a finite set $\Omega_0^\xi \subset \Omega^\xi$ and a sequence $(\lambda_\omega^\xi)_{\omega \in \Omega_0^\xi} \in \bbC$ such that 
    \begin{equation*}
        \sup_{\eta \in [-B,B]} \abs{ \sum_{\omega \in \Omega_0^\xi} \lambda_\omega^\xi \rme^{-2\pi (\eta-\omega)^2} - \rme^{2\pi\rmi \eta t} } < \epsilon.
    \end{equation*}
    According to equation~\eqref{eq:hebel_zalik}, we have 
    \begin{align*}
        \abs{\int_{-B}^B H_\xi(\eta) \rme^{2 \pi \rmi \eta t} \dd \eta} &\leq \int_{-B}^B \abs{ H_\xi(\eta) } \abs{ \sum_{\omega \in \Omega_0^\xi} \lambda_\omega^\xi \rme^{-2\pi (\eta-\omega)^2} - \rme^{2\pi\rmi \eta t} } \dd \eta \\
        &\leq \epsilon \cdot \norm{H_\xi}_1.
    \end{align*}
    Since the above holds for all $\epsilon > 0$, we conclude 
    \begin{equation*}
        \calF^{-1} H_\xi (t) = \int_{-B}^B H_\xi(\eta) \rme^{2 \pi \rmi \eta t} \dd \eta = 0, \qquad t \in \bbR.
    \end{equation*}
    Therefore, $H_\xi = 0$ in $L^1(\bbR)$ which implies that 
    \begin{equation}\label{eq:almost_done}
        \calF^{-1} f (\eta) \overline{ \calF^{-1} f (\eta - \xi) } = \calF^{-1} g (\eta) \overline{ \calF^{-1} g (\eta - \xi) }
    \end{equation}
    for almost all $\eta \in \bbR$.
    
    The rest of the proof is classical. We present the following argument inspired by \cite[Theorem~2.5 on p.~588]{auslander1985radar} for the convenience of the reader: another application of Fubini's theorem implies that 
    \begin{equation*}
        \xi \mapsto \calF^{-1} f (\eta) \overline{ \calF^{-1} f (\eta - \xi) } \in L^1(\bbR)
    \end{equation*}
    for almost all $\eta \in \bbR$ (and the same is true for $g$). If we fix such an $\eta$, then we can apply the Fourier transform to equation~\eqref{eq:almost_done} and obtain 
    \begin{equation*}
        \calF^{-1} f (\eta) \overline{ f(t) } = \calF^{-1} g (\eta) \overline{ g(t) }, \qquad t \in \bbR. 
    \end{equation*}
    If $f = 0$, then $\calF^{-1} g (\eta) \overline{ g(t) } = 0$ such that either $g = 0$ or $\calF^{-1} g = 0$ almost everywhere in which case $g = 0$. Hence, the theorem is proven if $f = 0$. So let us assume that $f \neq 0$. Then, there exists $t_0 \in \bbR$ such that $f(t_0) \neq 0$. Therefore, 
    \begin{equation*}
        \calF^{-1} f = \tau \cdot \calF^{-1} g, \qquad \tau := \frac{\overline{ g(t_0) }}{\overline{f(t_0)}}
    \end{equation*}
    in $L^1([-B,B])$ and thus $f = \tau g$. This proves the theorem since it implies $\abs{f(t_0)} = \abs{g(t_0)}$ which shows that $\abs{\tau} = 1$. 
\end{proof}

We can use the fractional Fourier transform to rotate our result in the time-frequency plane. To this end, we introduce the spaces 
\[
    \calF_\theta L^p([-B,B]) := \set{ f \in \calS'(\bbR) }{ f = \calF_\theta F \mbox{ for some } F \in L^p([-B,B]) },
\]
for $B>0$ and $p \in [1,\infty]$. These spaces are nested similarly to the Paley-Wiener spaces, with $\calF_\theta L^q([-B,B]) \subset \calF_\theta L^p([-B,B])$ for $1 \leq p \leq q \leq \infty$. Thus, for $p\geq 2$, we have $\calF_\theta L^p([-B,B])\subset \calF_\theta L^2([-B,B])\subset L^2(\bbR)$ since the fractional Fourier transform is unitary on $L^2(\bbR)$. Moreover, for $1 \leq p \leq 2$, we have $\calF_\theta L^p([-B,B])\subset L^q(\bbR)$, where $q$ is the H{\"o}lder conjugate of $p$, since the Hausdorff-Young inequality extends to the fractional Fourier transform \cite[Theorem~4.2 on p.~88]{chen2021fractional}.

With these spaces in hand, we can now state and prove a generalisation of our main result.

\begin{theorem}[Generalised main result]\label{thm:generalised_main_thm}
    Let $p \in [1,\infty]$, $B > 0$ and $\theta \in \bbR$. Let $X \subset \bbR$ be a set of uniqueness for $\mathrm{PW}_{2B}^1$ and let $\Omega \subset \bbR$ be a countable set such that $\sum_{\omega \in \Omega \setminus \{0\}} \abs{\omega}^{-1}$ diverges. Then, the following are equivalent for $f,g \in \calF_\theta L^p([-B,B])$:
    \begin{enumerate}
        \item $f = \rme^{\rmi \alpha} g$ for some $\alpha \in \bbR$,
        \item $\abs{\calG f} = \abs{\calG g}$ on $\operatorname{R}_\theta(\Omega \times X)$.
    \end{enumerate}
\end{theorem}

\begin{proof}
    Clearly, item~1 implies item~2. We will therefore assume that $\abs{\calG f} = \abs{\calG g}$ on $\operatorname{R}_\theta(\Omega \times X)$. According to Lemma~\ref{lem:frftandgabor}, we have $\abs{\calG \calF_{-\theta} f} = \abs{\calG \calF_{-\theta} g}$ on $\Omega \times X$. Note that $\calF_{-\theta} f, \calF_{-\theta} g \in L^p([-B,B])$ such that $\calF \calF_{-\theta} f, \calF \calF_{-\theta} g \in \mathrm{PW}_B^p$. Another application of Lemma~~\ref{lem:frftandgabor} yields $\abs{\calG \calF \calF_{-\theta} f} = \abs{\calG \calF \calF_{-\theta} g}$ on $X \times \Omega$. Therefore, Theorem~\ref{thm:mainthm} implies that $\calF \calF_{-\theta} f = \rme^{\rmi \alpha} \cdot \calF \calF_{-\theta} g$ for some $\alpha \in \bbR$. Finally, $f = \rme^{\rmi \alpha} g$ follows from Fourier inversion.
\end{proof}

\paragraph{Acknowledgements} The author would like to thank Rima Alaifari for her comments which have improved the presentation of the paper as well as the first reviewer whose remarks have inspired a generalisation of the main results. Additionally, the author acknowledges funding through the SNSF Grant 200021\_184698.

\paragraph{Declaration of generative AI and AI-assisted technologies in the writing process}

During the review process for this work, the author used ChatGPT to correct punctuation and orthography. After using this tool, the author reviewed and edited the content as needed and takes full responsibility for the content of the publication.

\bibliographystyle{plain}
\bibliography{sources} 

\end{document}